\documentclass[12pt,letterpaper]{amsart}
\usepackage{amssymb,amsmath,amscd,eucal,epsfig}
\usepackage{lscape}
\newcommand{\dis}{\displaystyle}

\def\<{\langle}                     
\def\>{\rangle}                     

\newcommand{\ben}{\begin{enumerate}}
\newcommand{\een}{\end{enumerate}}

\theoremstyle{plain}
\newtheorem{theorem}{Theorem}[section]

\newtheorem{remark}{Remark}[section]

\theoremstyle{definition}
\newtheorem{definition}{Definition}[section]

\numberwithin{equation}{section}
\begin{document}

\begin{center}
{\textbf{{Picture Fuzzy Multigroups}}}\ \\ \ \\

{Taiwo O. Sangodapo \\ toewuola77@gmail.com\\
Department of Mathematics,\\ University of Ibadan, Nigeria}
\end{center}
\ \\ \  \\

\noindent {\textbf{Abstract:}}
In this paper, picture fuzzy multisets were studied together with their associated properties. We also introduced the concept of picture fuzzy multigroups and established some of their algebraic properties.

\ \\ \ \\

\noindent \textbf{2010 AMS Classification:} 03E72, 08A72, 20N25

\ \\

\section{Introduction}
The origination of theory of fuzzy sets (FSs) was traced back to Zadeh's work where the crisp set was extended to fuzzy sets \cite{z}. The concept of intuitionistic fuzzy sets (IFSs) was pioneered by Atanassov \cite{a1}.
Cuong and Kreinovich \cite{ck} put forward the notion of picture fuzzy sets (PFSs) as a generalisation of fuzzy sets and intuitionistic fuzzy sets.

Yagar \cite{y} introduced fuzzy multisets (FMS) allowing each element occurring more that once. This notion was extended \cite{s} by Shinoj and Sunil to intuitionistic fuzzy multisets. Cao et al \cite{cfs} initiated picture fuzzy multisets (PFMSs) as a generalisation of FM and IFMS and also as an extension of PFS. In \cite{t1, ts} some properties of PFMSs were investigated. This notion has been extended to relations, and as a result some properties were established in \cite{t2, t3}. 

Rosenfeld \cite{r} generalised fuzzy sets to fuzzy groups (FGs). In \cite{s1}, Shinoj et al extended Rosenfeld to fuzzy multigroup, and Shinoj and Sunil \cite{s2} studied the work in \cite{s1} to introduce the concept of intuitionistic fuzzy multigroups. 
 
In this paper, the concept of picture fuzzy multigroups was introduced as an extension of the work in \cite{s2}. We also established some of the properties associated with them.

\ \\ \ \\

\section{Preliminaries}
\begin{definition} \cite{z}
Let $Y$ be a nonempty set. A FS $Q$ of $Y$ is an object of the form $$Q = \lbrace \langle y, \sigma_{Q}(x) \rangle | y \in Y \rbrace$$  with a membership function $$\sigma_{Q} : Q \longrightarrow [0, 1]$$ where the function $\sigma_{Q}(y)$ denotes the degree of membership of $y \in Q.$
\end{definition}

\begin{definition} \cite{a1}
Let a nonempty set $Y$ be fixed. An IFS $Q$ of $Y$ is an object of the form $$Q = \lbrace \langle y, \sigma_{Q}(y), \tau_{Q}(y) \rangle | y \in Y \rbrace$$ where the functions $$\sigma_{Q}: Y \rightarrow [0, 1]~\text{and}~ \tau_{Q}: Y \rightarrow [0, 1]$$ are called the membership and non-membership degrees of $y \in Q$, respectively, and for every $y \in Y$, $$0 \leq \sigma_{Q}(y) + \tau_{Q}(y) \leq 1.$$
\end{definition}

\begin{definition} \cite{a1}
Given a nonempty set $\mathcal{C}.$  An IFS $\mathcal{D}$ of $\mathcal{C}$ is written as $$\mathcal{D} = \lbrace \langle \dfrac{\sigma_{\mathcal{D}}(z), \tau_{\mathcal{D}}(z)}{z} \rangle | z \in \mathcal{C} \rbrace,$$
where the functions $\sigma_{\mathcal{D}}: \mathcal{C} \rightarrow \mathbb{I}~\text{and}~ \tau_{\mathcal{D}}: \mathcal{C} \rightarrow \mathbb{I}$ are called the membership and non-membership degrees of $z \in \mathcal{C}$, respectively, and for every $z \in \mathcal{C}$, $$0 \leq \sigma_{\mathcal{D}}(z) + \tau_{\mathcal{D}}(z) \leq 1.$$
\end{definition}

\begin{definition} \cite{ck}
A picture fuzzy set Q of $Y$ is defined as $$Q = \lbrace (y, \sigma_{Q}(y), \tau_{Q}(y), \gamma_{Q}(y))| y \in Y \rbrace,$$ where the functions $$\sigma_{Q}: Y \rightarrow [0, 1],~ \tau_{Q}: Y \rightarrow [0, 1]~ \text{and}~ \gamma_{Q}: Y \rightarrow [0, 1]$$
are called the positive, neutral and negative membership degrees of $y \in Q$, respectively, and $\sigma_{Q}, \tau_{Q}, \gamma_{Q}$ satisfy $$0 \leq \sigma_{Q}(y) + \tau_{Q}(y) + \gamma_{Q}(y) \leq 1,~ \forall y \in Y.$$ For each $y \in Y$, $S_{Q}(y) = 1 - (\sigma_{Q}(y) + \tau_{Q}(y) + \gamma_{Q}(y))$ is called the refusal membership degree of $y \in Q$.
\end{definition}

\begin{definition} \cite{ck}
Let $Q$ and $R$ be two PFSs. Then, the inclusion, equality, union, intersection and complement are defined as follow:
\begin{itemize}
\item[(i)] $Q \subseteq R$ if and only if for all $y \in Y$, $\sigma_{Q}(y) \leq \sigma_{R}(y)$, $\tau_{Q}(y) \leq \tau_{R}(y)$ and $\gamma_{Q}(y) \geq \gamma_{R}(y).$ \ \\
\item[(ii)] $Q = R$ if and only if $Q \subseteq R$ and $R \subseteq Q.$ \ \\
\item[(iii)] $Q \cup R = \lbrace ( y, \sigma_{Q}(y) \vee \sigma_R(y), \tau_Q(y) \wedge \tau_R(y)), \gamma_Q(y) \wedge \gamma_R(y)) | y \in Y \rbrace.$ \ \\
\item[(iv)] $Q \cap R = \lbrace ( y, \sigma_{Q}(y) \wedge \sigma_R(y), \tau_Q(y) \wedge \tau_R(y)), \gamma_Q(y) \vee \gamma_R(y)) | y \in Y \rbrace.$
\end{itemize}
\end{definition}
\noindent The set of all picture fuzzy multisets over $\mathcal{C},$ is denoted as PFMS($\mathcal{C}.$)

\begin{definition}\cite{dp}
Let $(G, \ast)$ be a crisp group and $Q = \lbrace (y, \sigma_{Q}(y), \tau_Q(y), \eta_Q(y))~|~ y \in G \rbrace$ be a PFS in $G$. Then, $Q$ is called picture fuzzy subgroup of $G$ (PFSG) if\\ (i)~ $\sigma_Q(a \ast b) \geq \sigma_Q(a) \wedge \sigma_Q(b),~ \tau_Q(a \ast b) \geq \tau_Q(a) \wedge \tau_Q(b),~ \eta_Q(a \ast b) \leq \eta_Q(a) \vee \eta_Q(b) $\\ \ \\ (ii) $\sigma_Q(a^{-1}) \geq \sigma_Q(a),~ \tau_Q(a^{-1}) \geq \tau_Q(a),~ \eta_Q(a^{-1}) \leq \eta_Q(a)$ for all $a, b \in G$.\\ Notice that $a^{-1}$ is the inverse of $a \in G$,\\ \ \\ or equivalently, $Q$ is a PFSG of $G$ if and only if\\ $\sigma_Q(a \ast b^{-1}) \geq \sigma_Q(a) \wedge \sigma_Q(b),~ \tau_Q(a \ast b^{-1}) \geq \tau_Q(a) \wedge \tau_Q(b),~ \eta_Q(a \ast b^{-1}) \leq \eta_Q(a) \vee \eta_Q(b) .$
\end{definition}

\begin{definition} \cite{dp}
Let $(G, \ast)$ be a crisp group and $Q = (\sigma_{Q}, \tau_Q, \eta_Q)$ be a PFSG of $G$. Then, for $a \in G$ the picture fuzzy left coset of $Q \in G$ is the PFS $aQ = (\sigma_{aQ},~ \tau_{aQ},~ \eta_{aQ})$ defined by $$\sigma_{aQ}(u) = \sigma_{Q} (a^{-1} \ast u),~ \tau_{aQ}(u) = \tau_{Q} (a^{-1} \ast u)~ \text{and}~  \eta_{aQ}(u) = \eta_{Q} (a^{-1} \ast u)$$ for all $ u \in G$.
\end{definition}

\begin{definition}\cite{dp}
Let $(G, \ast)$ be a crisp group and $Q = (\sigma_{Q}, \tau_Q, \eta_Q)$ be a PFSG of $G$. Then, for $a \in G$ the picture fuzzy right coset of $Q \in G$ is the PFS $Qa = (\sigma_{Qa},~ \tau_{Qa},~ \eta_{Qa})$ defined by $$\sigma_{Qa}(u) = \sigma_{Q} (u \ast a^{-1}),~ \tau_{Qa}(y) = \tau_{Q} (u \ast a^{-1})~ \text{and}~ \eta_{Qa}(u) = \eta_{Q} ( u \ast a^{-1})$$ for all $u \in G$.
\end{definition}

\begin{definition}\cite{dp}
Let $(G, \ast)$ be a crisp group and $Q = (\sigma_{Q}, \tau_Q, \eta_Q)$ be a PFSG of $G$. Then, $Q$ is called picture fuzzy normal subgroup (PFNSG) of $G$ if $$\sigma_{Qa}(y) = \sigma_{aQ}(y),~ \tau_{Qa}(y) = \tau_{aQ}(y),~ \eta_{Qa}(y) = \eta_{aQ}(y)$$ for all $a,~ y \in G$.
\end{definition}
\begin{remark}
Dogra and Pal established that PFSG of $G$ is normal if and only if\\ $(i)~ \sigma_{Q}(u^{-1} \ast a \ast u) = \sigma_{Q}(a)$\\ $(ii)~ \tau_{Q}(u^{-1} \ast a \ast u) = \tau_{Q}(a) $\\ $(ii)~ \eta_{Q}(u^{-1} \ast a \ast u) = \eta_{Q}(a).$ For all $a \in Q$ and $u \in G.$
\end{remark}

Cut set of picture fuzzy sets was introduced by Dutta and Ganju \cite{dg} but the definition did not capture the cut set very well. Thus, Dogra and Pal \cite{dp} corrected it in their paper.

\begin{definition} \cite{dp}
Let $Q  = \lbrace (x, \sigma_{Q}, \tau_{Q}, \gamma_{Q})| a \in Y \rbrace$ be PFS over the universe $Y$. Then, $(r,s,t)$-cut set of $Q$ is the crisp set in $Q$, denoted by $C_{r, s, t} (Q)$ and is defined by $$C_{r, s, t} (Q) = \left\lbrace a \in Y | \sigma_{Q}(a) \geq r,~\tau_{Q}(a) \geq s, \gamma_{Q}(a) \leq t  \right\rbrace$$ $r, s, t \in [0, 1]$ with the condition  $0 \leq r + s + t \leq 1.$
\end{definition}

\begin{theorem} \cite{dp} \label{1}
Let $(G, \ast)$ be a crisp group and $Q = (\sigma_{Q}, \tau_Q, \eta_Q)$ be a PFSG of $G$. Then, $Q$ is a PFSG if and only if $C_{r, s, t} (Q)$ is a crisp subgroup of $G$.
\end{theorem}

\begin{theorem} \cite{dg} \label{p}
If $Q$ and $R$ are two PFSs of a universe $Y$, then the following holds \begin{itemize}
\item [(i)]~ $C_{r, s, t} (Q) \subseteq C_{u, v, w} (Q)$ if $r \geq u,~ s \geq v,~ t \leq w.$  \ \\
\item [(ii)]~ $C_{1-s-t, s, t}(Q) \subseteq C_{r,s, t}(Q) \subseteq C_{r, 1-r-t, t}(Q).$ \ \\
\item [(iii)]~ $Q \subseteq R$ implies $C_{r, s, t} (Q) \subseteq C_{r, s, t} (R).$ \ \\
\item [(iv)]~ $C_{r, s, t} (Q \cap R) = C_{r, s, t} (Q) \cap C_{r, s, t} (R).$ \ \\
\item [(v)]~ $C_{r, s, t} (Q \cup R) \supseteq C_{r, s, t} (Q) \cup C_{r, s, t} (R).$ \ \\
\item [(vi)]~ $C_{r, s, t} (\cap Q_{i}) = \cap C_{r, s, t} (Q_{i}).$ \ \\
\item [(vii)]~ $ C_{1,0,0}(Q) = Y.$
\end{itemize}
\end{theorem}

\section{Picture Fuzzy Multigroups}
This section introduces picture fuzzy multigroups. Throughout this section, a group witha binary operation and an identity element $e$ is denoted by $C.$ Also, the set of all picture fuzzy multisets and picture fuzzy multigroups over $C$ are denoted by PFMS(C) and PFMG(C), respectively.
\ \\
\begin{definition}
Let $D \in PFMS(C).$ Then, $D^{-1}$ is defined as $$PCM_D^{-1}(g) = PCM_D(g^{-1}),~~ NeCM_D^{-1}(g) = NeCM_D(g^{-1})~~ \text{and}~~ NCM_D^{-1}(g) = NCM_D(g^{-1})$$
\end{definition}
\begin{definition}
Let $D, E \in PFMS(C).$ Then, the composite of $D$ and $E$ is defined as $$PCM_{D \circ E}(g_1) = \bigvee \lbrace PCM_D(g_2) \wedge PCM_E(g_3) :~g_2,g_3 \in C~ \text{and}~ g_2g_3 = g_1 \rbrace $$
$$NeCM_{D \circ E}(g_1) = \bigvee \lbrace NeCM_D(g_2) \wedge NeCM_E(g_3) :~g_2,g_3 \in C~ \text{and}~ g_2g_3 = g_1 \rbrace $$ and $$NCM_{D \circ E}(g_1) = \bigwedge \lbrace NCM_D(g_2) \vee NCM_E(g_3) :~g_2,g_3 \in C~ \text{and}~ g_2g_3 = g_1 \rbrace .$$
\end{definition}

\begin{theorem}
Let $D, E, D_i \in PFMS(C).$ Then, the following properties hold:
\begin{itemize}
\item [i.] $(D^{-1})^{-1} = D$
\item [ii.] $D \subseteq E~\Rightarrow~D^{-1} \subseteq E^{-1}$
\item [iii.] $ \dis \left(\bigcup^n_{i = 1} D_i\right)^{-1} = \bigcup^n_{i = 1} D_i^{-1}$
\item [iv.] $\dis \left(\bigcap^n_{i = 1} D_i\right)^{-1} = \bigcap^n_{i = 1} D_i^{-1}$
\item [v.] $(D \circ E) = E^{-1} \circ D^{-1}$
\item [vi.] 
\begin{eqnarray*}
PCM_{D \circ E}(g_1) & = & \bigvee_{g_2 \in C} \lbrace PCM_D(g_2) \wedge PCM_E(g_2^{-1}g_1)\rbrace~~ \forall~~ g_1 \in C\\
& = & \bigvee_{g_2 \in C} \lbrace PCM_D(g_1 g^{-1}_2) \wedge PCM_E(g_2)\rbrace~~ \forall~~ g_1 \in C
\end{eqnarray*}

\begin{eqnarray*}
NeCM_{D \circ E}(g_1) & = & \bigvee_{g_2 \in C} \lbrace NeCM_D(g_2) \wedge NeCM_E(g_2^{-1}g_1)\rbrace~~ \forall~~ g_1 \in C\\
& = & \bigvee_{g_2 \in C} \lbrace NeCM_D(g_1 g^{-1}_2) \wedge NeCM_E(g_2)\rbrace~~ \forall~~ g_1 \in C
\end{eqnarray*}
and 
\begin{eqnarray*}
NCM_{D \circ E}(g_1) & = & \bigwedge_{g_2 \in C} \lbrace NCM_D(g_2) \vee NeCM_E(g_2^{-1}g_1)\rbrace~~ \forall~~ g_1 \in C\\
& = & \bigwedge_{g_2 \in C} \lbrace NCM_D(g_1 g^{-1}_2) \vee NCM_E(g_2)\rbrace~~ \forall~~ g_1 \in C
\end{eqnarray*}
\end{itemize} 
\end{theorem}

\begin{proof}
i.
\begin{eqnarray*}
PCM_{(D^{-1})^{-1}}(g_1) & = & PCM_{(D^{-1})}(g^{-1})\\
& = & PCM_D(g^{-1})^{-1} \\
& = & PCM_D(g)~~ \forall~~ g \in C.~~ \text{Since}~~ C ~ \text{is~~ a~~ group}~~((g)^{-1})^{-1} = g
\end{eqnarray*}

\begin{eqnarray*}
NeCM_{(D^{-1})^{-1}}(g_1) & = & NeCM_{(D^{-1})}(g^{-1})\\
& = & NeCM_D(g^{-1})^{-1} \\
& = & NeCM_D(g)~~ \forall~~ g \in C.~~ \text{Since}~~ C ~ \text{is~~ a~~ group}~~((g)^{-1})^{-1} = g
\end{eqnarray*} and 

\begin{eqnarray*}
NCM_{(D^{-1})^{-1}}(g_1) & = & NCM_{(D^{-1})}(g^{-1})\\
& = & NCM_D(g^{-1})^{-1} \\
& = & NCM_D(g)~~ \forall~~ g \in C.~~ \text{Since}~~ C ~ \text{is~~ a~~ group}~~((g)^{-1})^{-1} = g
\end{eqnarray*}
Hence, $D = (D^{-1})^{-1}$\\

ii. Given that $D \subseteq E,$ this implies that, 
$$PCM_{D}(g^{-1}) \leq PCM_{E}(g^{-1})~~ \forall~~ g \in C$$
$$PCM_{D^{-1}}(g) \leq PCM_{E}(g),$$

$$NeCM_{D}(g^{-1}) \leq NeCM_{E}(g^{-1})~~ \forall~~ g \in C$$
$$NeCM_{D^{-1}}(g) \leq NeCM_{E}(g)$$ and 

$$NCM_{D}(g^{-1}) \geq NCM_{E}(g^{-1})~~ \forall~~ g \in C$$
$$NCM_{D^{-1}}(g) \geq NCM_{E}(g)$$ So, $D^{-1} \subseteq E^{-1}$\\

iii. \begin{eqnarray*}
PCM_{\dis \left(\bigcup^n_{i = 1} D_i\right)^{-1}}(g) & = & PCM_{\dis \left(\bigcup^n_{i = 1} D_i\right)}(g^{-1})\\
& = & \bigvee \lbrace  PCM_{D_i}(g^{-1});~~ i = 1, 2, \cdots, n \rbrace~~ \text{by~Definition ?} \\
& = & \bigvee \lbrace  PCM_{D_i}^{-1}(g);~~ i = 1, 2, \cdots, n \rbrace\\
& = & PCM_{\dis \left(\bigcup^n_{i = 1} {D_i}^{-1}\right)}(g)~~ \text{by~Definition ?}
\end{eqnarray*}

\begin{eqnarray*}
NeCM_{\dis \left(\bigcup^n_{i = 1} D_i\right)^{-1}}(g) & = & NeCM_{\dis \left(\bigcup^n_{i = 1} D_i\right)}(g^{-1})\\
& = & \bigvee \lbrace  NeCM_{D_i}(g^{-1});~~ i = 1, 2, \cdots, n \rbrace~~ \text{by~Definition ?} \\
& = & \bigvee \lbrace  NeCM_{D_i}^{-1}(g);~~ i = 1, 2, \cdots, n \rbrace\\
& = & NeCM_{\dis \left(\bigcup^n_{i = 1} {D_i}^{-1}\right)}(g)~~ \text{by~Definition ?}
\end{eqnarray*}
and

\begin{eqnarray*}
NCM_{\dis \left(\bigcap^n_{i = 1} D_i\right)^{-1}}(g) & = & NCM_{\dis \left(\bigcap^n_{i = 1} D_i\right)}(g^{-1})\\
& = & \bigwedge \lbrace  NCM_{D_i}(g^{-1});~~ i = 1, 2, \cdots, n \rbrace~~ \text{by~Definition ?} \\
& = & \bigwedge \lbrace  NCM_{D_i}^{-1}(g);~~ i = 1, 2, \cdots, n \rbrace\\
& = & NCM_{\dis \left(\bigcap^n_{i = 1} {D_i}^{-1}\right)}(g)~~ \text{by~Definition ?}
\end{eqnarray*}
Therefore, $ \dis \left(\bigcup^n_{i = 1} D_i\right)^{-1} = \bigcup^n_{i = 1} D_i^{-1}.$\\

iv. Property iv can be done in a similar manner.\\

v.
\begin{eqnarray*}
PCM_{ (D \circ E)^{-1}}(g_1) & = & PCM_{ (D \circ E)}(g^{-1}_1)\\
& = & \bigvee \lbrace  PCM_{D}(g_2) \wedge PCM_{E}(g_3);~~ g_2, g_3 \in C,~~ g_2g_3 = g_1^{-1} \rbrace \\
& = & \bigvee \lbrace  PCM_{E}(g_3) \wedge PCM_{D}(g_2);~~ g_2, g_3 \in C,~~ (g_2g_3)^{-1} = g_1 \rbrace\\
& = & \bigvee \lbrace  PCM_{E}(g^{-1}_3)^{-1} \wedge PCM_{D}(g^{-1}_2)^{-1};~~ g_2, g_3 \in C,~~ (g_2g_3)^{-1} = g_1 \rbrace\\
& = & \bigvee \lbrace  PCM_{E^{-1}}(g^{-1}_3) \wedge PCM_{D^{-1}}(g^{-1}_2);~~ g^{-1}_2, g^{-1}_3 \in C,~~ g^{-1}_2g^{-1}_3 = g_1 \rbrace\\
& = & PCM_{E^{-1} \circ D^{-1}}(g_1)~~ \forall~~ g_1 \in C.
\end{eqnarray*}

\begin{eqnarray*}
NeCM_{ (D \circ E)^{-1}}(g_1) & = & NeCM_{ (D \circ E)}(g^{-1}_1)\\
& = & \bigvee \lbrace  NeCM_{D}(g_2) \wedge NeCM_{E}(g_3);~~ g_2, g_3 \in C,~~ g_2g_3 = g_1^{-1} \rbrace \\
& = & \bigvee \lbrace  NeCM_{E}(g_3) \wedge NeCM_{D}(g_2);~~ g_2, g_3 \in C,~~ (g_2g_3)^{-1} = g_1 \rbrace\\
& = & \bigvee \lbrace  NeCM_{E}(g^{-1}_3)^{-1} \wedge NeCM_{D}(g^{-1}_2)^{-1};~~ g_2, g_3 \in C,~~ (g_2g_3)^{-1} = g_1 \rbrace\\
& = & \bigvee \lbrace  NeCM_{E^{-1}}(g^{-1}_3) \wedge NeCM_{D^{-1}}(g^{-1}_2);~~ g^{-1}_2, g^{-1}_3 \in C,~~ g^{-1}_2g^{-1}_3 = g_1 \rbrace\\
& = & NeCM_{E^{-1} \circ D^{-1}}(g_1)~~ \forall~~ g_1 \in C.
\end{eqnarray*}
and 

\begin{eqnarray*}
NCM_{ (D \circ E)^{-1}}(g_1) & = & NCM_{ (D \circ E)}(g^{-1}_1)\\
& = & \bigwedge \lbrace  NCM_{D}(g_2) \vee NCM_{E}(g_3);~~ g_2, g_3 \in C,~~ g_2g_3 = g_1^{-1} \rbrace \\
& = & \bigwedge \lbrace  NCM_{E}(g_3) \vee NCM_{D}(g_2);~~ g_2, g_3 \in C,~~ (g_2g_3)^{-1} = g_1 \rbrace\\
& = & \bigwedge \lbrace  PCM_{E}(g^{-1}_3)^{-1} \vee NCM_{D}(g^{-1}_2)^{-1};~~ g_2, g_3 \in C,~~ (g_2g_3)^{-1} = g_1 \rbrace\\
& = & \bigwedge \lbrace  NCM_{E^{-1}}(g^{-1}_3) \vee NCM_{D^{-1}}(g^{-1}_2);~~ g^{-1}_2, g^{-1}_3 \in C,~~ g^{-1}_2g^{-1}_3 = g_1 \rbrace\\
& = & NCM_{E^{-1} \circ D^{-1}}(g_1)~~ \forall~~ g_1 \in C.
\end{eqnarray*}
Hence, $(D \circ E) = E^{-1} \circ D^{-1}.$\\

vi. Since $C$ is a group, so for each $g_1, g_2 \in C,$ there exists a unique $g_3 (= g^{-1}_2 g_1) \in C$ such that $g_2 g_3 = g_1.$ Thus, 
\begin{eqnarray*}
PCM_{D \circ E}(g_1) & = & \dis \bigvee_{g_2 \in C} \lbrace  PCM_{D}(g_2) \wedge PCM_{E}(g^{-1}_2 g_1) \rbrace~~ \forall g_1 \in C, \\
NeCM_{D \circ E}(g_1) & = & \dis \bigvee_{g_2 \in C} \lbrace  NeCM_{D}(g_2) \wedge NeCM_{E}(g^{-1}_2 g_1) \rbrace~~ \forall g_1 \in C, \\ 
NCM_{D \circ E}(g_1) & = & \dis \bigwedge_{g_2 \in C} \lbrace  NCM_{D}(g_2) \vee NCM_{E}(g^{-1}_2 g_1) \rbrace~~ \forall g_1 \in C.
\end{eqnarray*}
Also,
\begin{eqnarray*}
PCM_{D \circ E}(g_1) & = & \dis \bigvee_{g_2 \in C} \lbrace  PCM_{E}(g_3) \wedge PCM_{D}(g_2);~~ g_2, g_3 \in C ~~\text{and}~~ g_2 g_3 = g_1 \rbrace, \\
NeCM_{D \circ E}(g_1) & = & \dis \bigvee_{g_2 \in C} \lbrace  NeCM_{E}(g_3) \wedge NeCM_{D}(g_2);~~ g_2, g_3 \in C ~~\text{and}~~ g_2 g_3 = g_1 \rbrace, \\ 
NCM_{D \circ E}(g_1) & = & \dis \bigwedge_{g_2 \in C} \lbrace  NCM_{E}(g_3) \vee NCM_{D}(g_2);~~ g_2, g_3 \in C ~~\text{and}~~ g_2 g_3 = g_1 \rbrace.
\end{eqnarray*}
Since $C$ is a group, it follows that for each $g_1, g_2 \in C$ there exists a unique $g_3 (= g^{-1}_2 g_1) \in C$ such that $g_3 g_2 = g_1.$ Then, 

\begin{eqnarray*}
PCM_{D \circ E}(g_1) & = & \dis \bigvee_{g_2 \in C} \lbrace  PCM_{E}(g_1g_2^{-1}) \wedge PCM_{D}(g_2) \rbrace \forall~~ g_1 \in C, \\
NeCM_{D \circ E}(g_1) & = & \dis \bigvee_{g_2 \in C} \lbrace  NeCM_{E}(g_1g_2^{-1}) \wedge NeCM_{D}(g_2) \rbrace \forall~~ g_1 \in C,\\ 
NCM_{D \circ E}(g_1) & = & \dis \bigwedge_{g_2 \in C} \lbrace  NCM_{E}(g_1g_2^{-1}) \vee NCM_{D}(g_2) \rbrace \forall~~ g_1 \in C.
\end{eqnarray*}
\end{proof}

\begin{definition}
Let $C$ be a group. Then, a picture fuzzy multiset $D$ over $C$ is called a picture fuzzy multigroup (PFMG) over $C$ if the positive, neutral and negative counts memberships of $D$ satisfy the following, for all $g_1, g_2 \in C$
\begin{itemize}
\item [i.] $PCM_{D}(g_1 g_2) \geq PCM_{D}(g_1) \wedge PCM_{D}(g_2)$
\item [ii.] $PCM_{D}(g^{-1}_1) \geq PCM_{D}(g_1)$ 
\item [iii.] $NeCM_{D}(g_1 g_2) \geq NeCM_{D}(g_1) \wedge NeCM_{D}(g_2)$
\item [iv.] $NeCM_{D}(g^{-1}_1) \geq NeCM_{D}(g_1)$
\item [v.] $NCM_{D}(g_1 g_2) \leq NCM_{D}(g_1) \vee NCM_{D}(g_2)$
\item [vi.] $NCM_{D}(g^{-1}_1) \leq NCM_{D}(g_1)$
\end{itemize}  
\end{definition}

\begin{theorem}
Let $D \in PFMS(C)$ and  $PCM_{D}(g^{-1}) \geq PCM_{D}(g),$ $NeCM_{D}(g^{-1}) \geq NeCM_{D}(g)$ and $NCM_{D}(g^{-1}) \leq NCM_{D}(g).$ Then, $L(g;~ D) = L(g^{-1};~ D).$
\end{theorem}

\begin{proof}
From the theorem, $PCM_{D}(g^{-1}) \geq PCM_{D}(g),$ $NeCM_{D}(g^{-1}) \geq NeCM_{D}(g)$ and $NCM_{D}(g^{-1}) \leq NCM_{D}(g).$\\

So, $PCM_{D}(g) = PCM_{D}((g^{-1})^{-1}) \geq PCM_{D}(g^{-1}).$ \\ Then, $PCM_{D}(g) = PCM_{D}(g^{-1})$ \\

$NeCM_{D}(g) = NeCM_{D}((g^{-1})^{-1}) \geq NeCM_{D}(g^{-1}).$ \\ Then, $NeCM_{D}(g) = NeCM_{D}(g^{-1})$ and \\

$NCM_{D}(g) = NCM_{D}((g^{-1})^{-1}) \leq NCM_{D}(g^{-1}).$ \\ Then, $NCM_{D}(g) = NCM_{D}(g^{-1}).$\\ 

Thus, $L(g ; D) = |PCM_{D}(g)| = |NeCM_{D}(g)| = |NCM_{D}(g)|$ by Definition ?\\ Therefore, $$L(g;~ D) = |PCM_{D}(g^{-1})| = |NeCM_{D}(g^{-1})| = |NCM_{D}(g^{-1})| = L(g^{-1};~ D).$$ 
\end{proof}

\begin{theorem} \label{1}
Let $D \in PFMS(C).$ Then, for all $g \in C,$ \begin{itemize}
\item [i.] $PCM_{D}(e) \geq PCM_{D}(g)$ 
\item [ii.] $NeCM_{D}(e) \geq NeCM_{D}(g)$
\item [iii.] $NCM_{D}(e) \leq NCM_{D}(g)$
\item [iv.] $PCM_{D}(g^{n}) \geq PCM_{D}(g)$
\item [v.] $NeCM_{D}(g^n) \geq NeCM_{D}(g)$
\item [vi.] $NCM_{D}(g^n) \leq NCM_{D}(g)$
\item [vii.] $D^{-1} \supseteq D.$
\end{itemize}
\end{theorem}

\begin{proof}
Let $g, g^{-1} \in C.$\\
i. 
\begin{eqnarray*}
PCM_{D}(e) & = & PCM_{D}(g g^{-1})\\
& \geq & PCM_{D}(g) \wedge PCM_{D}(g^{-1}) \\
& \geq & PCM_{D}(g) \wedge PCM_{D}(g)\\
& = & PCM_{D}(g)
\end{eqnarray*}

ii \begin{eqnarray*}
NeCM_{D}(e) & = & NeCM_{D}(g g^{-1})\\
& \geq & NeCM_{D}(g) \wedge NeCM_{D}(g^{-1}) \\
& \geq & NeCM_{D}(g) \wedge NeCM_{D}(g)\\
& = & NeCM_{D}(g)
\end{eqnarray*}

iii. \begin{eqnarray*}
NCM_{D}(e) & = & NCM_{D}(g g^{-1})\\
& \leq & NCM_{D}(g) \vee NCM_{D}(g^{-1}) \\
& \leq & NCM_{D}(g) \vee NCM_{D}(g)\\
& = & NCM_{D}(g)
\end{eqnarray*}

iv. \begin{eqnarray*}
PCM_{D}(g^n) & \geq & PCM_{D}(g^{n-1}) \wedge PCM_{D}(g)\\
& \geq & PCM_{D}(g) \wedge PCM_{D}(g) \wedge \cdots \wedge PCM_{D}(g)~~ (\text{recursively}) \\
& = & PCM_{D}(g)
\end{eqnarray*}

v.  \begin{eqnarray*}
NeCM_{D}(g^n) & \geq & NeCM_{D}(g^{n-1}) \wedge NeCM_{D}(g)\\
& \geq & NeCM_{D}(g) \wedge NeCM_{D}(g) \wedge \cdots \wedge NeCM_{D}(g)~~ (\text{recursively}) \\
& = & NeCM_{D}(g)
\end{eqnarray*}

vi.  \begin{eqnarray*}
NCM_{D}(g^n) & \leq & NCM_{D}(g^{n-1}) \vee NCM_{D}(g)\\
& \leq & NCM_{D}(g) \vee NCM_{D}(g) \vee \cdots \vee NCM_{D}(g)~~ (\text{recursively}) \\
& = & NCM_{D}(g)
\end{eqnarray*}

vii. $$PCM_{D^{-1}}(g) = PCM_{D}(g^{-1}) \geq PCM_D(g)$$
$$NeCM_{D^{-1}}(g) = NeCM_{D}(g^{-1}) \geq NeCM_D(g)$$
$$NCM_{D^{-1}}(g) = NCM_{D}(g^{-1}) \leq NCM_D(g)$$ Thus, $D^{-1} \supseteq D.$
\end{proof}

\begin{theorem}
Let $D \in PFMS(C).$ Then, $D \in PFMG(C)$ if and only if $$(i)~PCM_{D}(g_1 g_2^{-1}) \geq PCM_{D}(g_1) \wedge PCM_{D}(g_2^{-1}),$$ $$(ii)~NeCM_{D}(g_1 g_2^{-1}) \geq NeCM_{D}(g_1) \wedge NeCM_{D}(g_2^{-1})$$ and $$(iii)~NCM_{D}(g_1 g_2^{-1}) \leq NCM_{D}(g_1) \vee NCM_{D}(g_2^{-1})$$ for all $g_1, g_2 \in C.$
\end{theorem}

\begin{proof}
Suppose that $D \in PFMG(C).$ This implies that, for all $g_1, g_2 \in C$
\begin{eqnarray*}
PCM_{D}(g_1 g_2^{-1}) & \geq & PCM_{D}(g_1) \wedge PCM_{D}(g_2^{-1}) \\
& \geq & PCM_{D}(g_1) \wedge PCM_{D}(g_2)
\end{eqnarray*}

\begin{eqnarray*}
NeCM_{D}(g_1 g_2^{-1}) & \geq & NeCM_{D}(g_1) \wedge NeCM_{D}(g_2^{-1}) \\
& \geq & NeCM_{D}(g_1) \wedge NeCM_{D}(g_2)
\end{eqnarray*}

\begin{eqnarray*}
NCM_{D}(g_1 g_2^{-1}) & \leq & NCM_{D}(g_1) \vee NCM_{D}(g_2^{-1}) \\
& \leq & NCM_{D}(g_1) \vee NCM_{D}(g_2)
\end{eqnarray*}
Conversely, suppose that (i), (ii) and (iii) hold. Also, 
\begin{eqnarray*}
PCM_{D}(g_1^{-1}) & = & PCM_{D}(e g_1^{-1})\\
& \geq & PCM_{D}(e) \wedge PCM_{D}(g_1^{-1}) \\
& \geq & PCM_{D}(e) \wedge PCM_{D}(g_1)\\
& = & PCM_{D}(g_1),
\end{eqnarray*}

\begin{eqnarray*}
NeCM_{D}(g_1^{-1}) & = & NeCM_{D}(e g_1^{-1})\\
& \geq & NeCM_{D}(e) \wedge NeCM_{D}(g_1^{-1}) \\
& \geq & NeCM_{D}(e) \wedge NeCM_{D}(g_1)\\
& = & NeCM_{D}(g_1)
\end{eqnarray*}
and 
\begin{eqnarray*}
NCM_{D}(g_1^{-1}) & = & NCM_{D}(e g_1^{-1})\\
& \leq & NCM_{D}(e) \vee NCM_{D}(g_1^{-1}) \\
& \leq & NCM_{D}(e) \vee NCM_{D}(g_1)\\
& = & NCM_{D}(g_1)
\end{eqnarray*}
Now, \begin{eqnarray*}
PCM_{D}(g_1 g_2) & \geq & PCM_{D}(g_1) \wedge PCM_{D}(g_2^{-1})\\
& = & PCM_{D}(g_1) \wedge PCM_{D}(g_2) 
\end{eqnarray*}
\begin{eqnarray*}
NeCM_{D}(g_1 g_2) & \geq & NeCM_{D}(g_1) \wedge NeCM_{D}(g_2^{-1})\\
& = & NeCM_{D}(g_1) \wedge NeCM_{D}(g_2) 
\end{eqnarray*} and
\begin{eqnarray*}
NCM_{D}(g_1 g_2) & \leq & NCM_{D}(g_1) \vee NCM_{D}(g_2^{-1})\\
& = & NCM_{D}(g_1) \vee NCM_{D}(g_2) 
\end{eqnarray*}
Therefore, $D \in PFMG(C).$
\end{proof}

\begin{definition}
Let $D \in PFMS(C).$ Then, the $(d, r, s, t)$-cut set of $D$ denoted by $D[r, s, t, d]$ is defined as $$D[r, s, t, d] = \lbrace c \in C~|~\sigma_D^k(c)\geq r,~\tau_D^k(c) \geq s,~\eta_D^k(c) \leq t;~ L(c)\geq k \geq d~\text{and}~k, d \in \mathbb{N}  \rbrace.$$
\end{definition}

\begin{theorem}
Let $D \in PFMS(C).$ Then, $D[r, s, t, d]$ is a subgroup of $C.$
\end{theorem}

\begin{proof}
Let $g_1, g_2 \in D[r, s, t, d],$\\ $\Rightarrow$ for $k \geq d,$ we have $$\sigma_D^k(g_1)\geq r~~ \text{and}~~ \sigma_D^k(g_2)\geq r;$$ $$\tau_D^k(g_1)\geq s~~ \text{and}~~ \tau_D^k(g_2)\geq s$$ and $$\eta_D^k(g_1)\leq t~~ \text{and}~~ \eta_D^k(g_2)\leq t.$$ Then, $$\sigma_D^k(g_1 g_2^{-1})\geq r, ~~ \tau_D^k(g_1 g_2^{-1})\geq s~~\text{and}~~ \eta_D^k(g_1 g_2^{-1})\leq t$$ This implies that if $g_1, g_2 \in D[r, s, t, d]$ then, $g_1 g_2^{-1} \in D[r, s, t, d].$\\ Therefore, $D[r, s, t, d]$ is a subgroup of $C.$
\end{proof}

\begin{definition}
Let $D \in PFMS(C).$ Then, $D^{\ast}$ is defined as $$D^{\ast} = \lbrace g \in G~|~PCM_D(g) = PCM_D(e),~~ NeCM_D(g) = NeCM_D{e}~~ \text{and}~~ NCM_D(g) = NCM_D(e) \rbrace.$$
\end{definition}

\begin{theorem}
Let $D \in PFMS(C).$ Then, $D^{\ast}$ is a subgroup of $C.$
\end{theorem}

\begin{proof}
Let $g_1, g_2 \in D^{\ast}.$\\ Then, $$PCM_D(g_1) = PCM_D(g_2) = PCM_D(e)\quad \quad \quad \quad \quad \quad \quad \quad \quad(i)$$ 
$$NeCM_D(g_1) = NeCM_D(g_2) = NeCM_D(e) \quad \quad \quad \quad \quad \quad \quad \quad (ii)$$
$$NCM_D(g_1) = NCM_D(g_2) = NCM_D(e)\quad \quad \quad \quad \quad \quad \quad \quad \quad(iii)$$
Then, we have the following;

\begin{eqnarray*}
PCM_{D}(g_1 g_2^{-1}) & \geq & PCM_{D}(g_1) \wedge PCM_{D}(g_2^{-1}) \\
& \geq & PCM_{D}(e) \wedge PCM_{D}(e)~~\text{by~~ i}\\
& = & PCM_{D}(e)
\end{eqnarray*}
But, $PCM_{D}(g_1 g_2^{-1}) \geq PCM_{D}(e)$\\ i.e $PCM_{D}(g_1 g_2^{-1}) = PCM_{D}(e),$

\begin{eqnarray*}
NeCM_{D}(g_1 g_2^{-1}) & \geq & NeCM_{D}(g_1) \wedge NeCM_{D}(g_2^{-1}) \\
& \geq & NeCM_{D}(e) \wedge NeCM_{D}(e)\\
& = & NeCM_{D}(e)
\end{eqnarray*}
But, $NeCM_{D}(g_1 g_2^{-1}) \geq NeCM_{D}(e)$\\ i.e $NeCM_{D}(g_1 g_2^{-1}) = NeCM_{D}(e),$

\begin{eqnarray*}
NCM_{D}(g_1 g_2^{-1}) & \leq & NCM_{D}(g_1) \vee NCM_{D}(g_2^{-1}) \\
& \leq & NCM_{D}(e) \vee NCM_{D}(e)\\
& = & NCM_{D}(e)
\end{eqnarray*}
But, $NCM_{D}(g_1 g_2^{-1}) \leq NCM_{D}(e)$\\ i.e $NCM_{D}(g_1 g_2^{-1}) = NCM_{D}(e).$\\
This means that $g_1 g_2^{-1} \in D^{\ast}.$\\ Therefore, $D^{\ast}$ is a subgroup of $C.$
\end{proof}

\ \\ \ \\


\end{document}